\documentclass{amsart}
\usepackage{amsmath,amscd,amssymb, color}
\usepackage{mathtools, xy, hyperref}

\theoremstyle{plain}
\newtheorem{theorem}{Theorem}[section]

\newtheorem{prop}[theorem]{Proposition}

\theoremstyle{definition}
\newtheorem{remark}[theorem]{Remark}

\newcommand{\OO}{{\mathcal O}}

\newcommand{\A}{{\mathbb A}}

\newcommand{\V}{{\mathbb V}}

\newcommand{\Z}{\mathbb{Z}}

\newcommand{\C}{\mathbb{C}}
\newcommand{\SC}{\mathcal{C}}

\renewcommand{\P}{\mathbb{P}}

\newcommand{\GL}{\mathrm{GL}}

\renewcommand{\phi}{\varphi}

\DeclareMathOperator{\Hilb}{Hilb}
\DeclareMathOperator{\Sym}{Sym}
\DeclareMathOperator{\Pic}{Pic}

\DeclareMathOperator{\PGL}{PGL}
\DeclareMathOperator{\SL}{SL}
\DeclareMathOperator{\Gr}{Gr}

\DeclareMathOperator{\Spec}{Spec}

\DeclareMathOperator{\Tr}{tr}

\newcommand{\np}{\mathrm{np}}

\xyoption{all}

\begin{document}

\title[{On some Hilbert schemes}]{On the equations defining some Hilbert schemes}

\author{Jonathan D. Hauenstein}
\address{Department of Applied and Computational Mathematics and Statistics, University of Notre Dame, Notre Dame, IN 46556, USA}
\email{hauenstein@nd.edu}
\author{Laurent Manivel}
\address{Institut de Math\'ematiques de Toulouse; UMR 5219, Universit\'e de Toulouse \& CNRS, F-31062 Toulouse Cedex 9, France}
\email{laurent.manivel@math.cnrs.fr}
\author{Bal\'azs Szendr\H oi}
\address{Mathematical Institute, University of Oxford, OX1 2DL Oxford, UK}
\email{szendroi@maths.ox.ac.uk}

\begin{abstract} We work out details of the extrinsic geometry 
for two Hilbert schemes of some contemporary interest: the Hilbert scheme $\Hilb^2 \P^2$ of two points on $\P^2$ and the dense open set parametrizing non-planar clusters in
the punctual Hilbert scheme $\Hilb^4_0(\A^3)$ of clusters of length four on~$\A^3$ with support at the origin. We find explicit equations in projective, respectively affine, embeddings for these spaces. In particular, we answer a question of Bernd Sturmfels
who asked for a description of the latter space that is amenable to further computations. While the explicit equations we find are controlled in a precise way by the representation theory of $\SL_3$, our arguments also rely on computer algebra. 
\end{abstract}

\maketitle

\section*{Introduction} 
\noindent The Hilbert scheme $\Hilb^2 \P^2$ of two points on the projective plane has a  projective embedding 
\begin{equation}\Hilb^2 \P^2  \hookrightarrow \P^{14}\label{emb1}\end{equation}
defined as a composition of a natural map into $\Gr(2,6)$  followed by the Pl\"ucker embedding (for details, see Section~\ref{sec:1.1}). 
On the other hand, the Hilbert scheme $\Hilb^4(\A^3)$ of $4$ points on affine $3$-space has a distinguished affine open subset $\Hilb^4(\A^3)_\np$ of non-planar clusters, which has an affine embedding
\[ \Hilb^4(\A^3)_\np \hookrightarrow \A^{15},\]
whose image is the cone over the same Grassmannian (see Section~\ref{sec:2.1}). This latter Hilbert scheme has a natural closed subset, the space  $\Hilb^4_0(\A^3)_\np$ of non-planar, punctual clusters, where punctual means that the scheme-theoretic support of the subscheme parametrized is at a single point (chosen to be the origin; see Section~\ref{sec:2.2} for details). We thus obtain an affine embedding
\begin{equation}\label{emb2}
\Hilb^4_0(\A^3)_\np \hookrightarrow \A^{15}.
\end{equation}
The connection between these Hilbert schemes goes back to Tikhomirov's \cite[Thm 3]{Tik}: the affine cone over $\Hilb^2 \P^2\subset\P^{14}$ is the singular locus of $\Hilb^4_0(\A^3)_\np$. 

Our aim in this paper is twofold: we give explicit equations to describe the images of the embeddings~\eqref{emb1} and~\eqref{emb2}, and we recover this relationship between the spaces. We begin in Section~\ref{sec:1.2} by identifying various spaces of interest as $\PGL_3$-orbit closures in $\Gr(2,6)$. Their defining polynomials arise from the representation theory of $\SL_3$ (see Section~\ref{sec:1.3}). We use computer algebra to derive specific polynomials, as well as to check various properties of the resulting systems of equations, in particular that they define reduced ideals. We also provide a more synthetic way to derive the equations in Section~\ref{sec:1.4}. 

Our main result is Theorem~\ref{main_thm} in Section~\ref{sec:2.2}. Section~\ref{sec:2.3} presents an explicit calculation to illuminate, and yet again reprove, some of the main results. The explicit polynomials defining the embeddings~\eqref{emb1} and~\eqref{emb2} are listed in the Appendix.

\section{The Hilbert scheme of two points on the projective plane}\label{sec:1}

\subsection{Basics on the Hilbert scheme of two points on the projective plane}\label{sec:1.1}

Let $U$ be a $3$-dimensional vector space. In this section, we recall some well-known facts about the Hilbert scheme $\Hilb^2 \P(U)$ of two points on the projective plane $\P(U)$. We refer to~\cite[Sections 2 and 3]{ABCH} for more details and further references.

Let $h\colon\Hilb^2 \P(U) \to \Sym^2 \P(U)$ be the Hilbert--Chow morphism, $H=h^*({\mathcal O}(1))$
the pullback of the natural generator of $\Pic(\Sym^2 \P(U))$. Let $\Delta\subset \Sym^2 \P(U)$ be the diagonal and $B=h^{-1}(\Delta)$ the exceptional locus of $h$, a threefold ruled over $\Delta\cong\P(U)$. It is known that 
\[\Pic(\Hilb^2 \P(U))\cong \Z[H]\oplus \Z\left[\frac{B}{2}\right].
\] 
Let $\mathcal{L}$ denote the line bundle 
$\OO(2)$ on $\P(U)$. There is an associated rank two bundle $\mathcal{L}^{[2]}$ on $\Hilb^2\P(U)$ whose fiber over a length $2$ subscheme $\zeta\subset\P(U)$ is $H^0(\zeta,\mathcal{L})$. Its space of sections is $$H^0(\Hilb^2\P(U),\mathcal{L}^{[2]})\cong H^0(\P(U),\mathcal{L})=S^2U^*$$ and it is globally generated. 
Its determinant 
\[D=\det(\mathcal{L}^{[2]}) = 2[H]-\left[\frac{B}{2}\right]\in\Pic(\Hilb^2\P(U))\] 
is very ample and yields
a $\PGL(U)$-equivariant chain of embeddings
\[ \varphi_{D}\colon \Hilb^2\P(U)\hookrightarrow \Gr(4,S^2U^*)=\Gr(2,S^2U) \hookrightarrow \P(\wedge^2S^2 U)\cong\P^{14}.
\] 
The image of $\varphi_{D}$ is of degree 
\[D^4= 16 H^4 - 16 H^3B + 6 H^2B^2 - H B^3 + \frac{B^4}{16}=21\]
which arises from the known intersection numbers~\cite{Ryan} on the Hilbert scheme, namely
\[(H^4,H^3B,H^2B^2,HB^3,B^4) = (3,0,-8,-24,-48).\]
The embedding $\varphi_{D}$ restricts to a $\PGL(U)$-equivariant chain of embeddings 
\[\varphi_{D}|_B\colon B\hookrightarrow   \Gr(2,S^2 U)\hookrightarrow \P(\wedge^2S^2 U) \cong \P^{14}
\] 
of degree 
\[  (D|_B)^3= D^3B = 8 H^3 B - 6 H^2B^2 +  \frac{3 HB^3}{2}- \frac{B^4}{8}=18.
\] 

\subsection{Pencils of conics and orbit structure}\label{sec:1.2}

\label{subsec_pencils}
Consider the space
$\Gr(2,S^2 U)$
parametrizing pencils of conics on $U^*$. The decomposition of this space
into $\PGL(U)$-orbits is classical, see e.g.~\cite{WK}. There are eight types with 
normal forms as in the following table, yielding a decomposition of the Grassmannian $\Gr(2,S^2 U)$ into $\PGL(U)$-orbits $\OO_i$ (and sometimes $\OO'_i$) of dimension~$i$. 
$$\begin{array}{ll}
\OO_3 \qquad & \langle x^2,xy\rangle \\
\OO_4 & \langle x^2,y^2\rangle \\
\OO'_4 & \langle xy,xz\rangle \\
\OO_5 & \langle x^2, y^2+xz\rangle \\
\OO_6 & \langle x^2, yz\rangle \\
\OO'_6 & \langle x^2+yz, xz\rangle \\
\OO_7 & \langle x^2+y^2, xz\rangle \\
\OO_8 & \langle x^2+y^2, x^2+z^2\rangle 
\end{array}$$

\medskip
Let $Y_i, Y_i'$ denote the closure of $\OO_i, \OO_i'$ in $\Gr(2,S^2 U)$.  Clearly, $Y_8= \Gr(2,S^2 U)$. 
The Hasse diagram, where $\OO_i$ is connected to $\OO_j$ if the latter is open in $Y_i\setminus\OO_i$, 
is the following:
$$\xymatrix{
 & \OO_8 \ar@{->}[d]  &  \\  & \OO_7 \ar@{->}[rd] \ar@{->}[ld]  &  \\
 \OO_6 \ar@{->}[dr] & &\OO'_6 \ar@{->}[dl]\ar@{->}[dd]  \\ & \OO_5 \ar@{->}[dl]  &  \\
 \OO_4 \ar@{->}[dr]  & &\OO'_4 \ar@{->}[dl] \\ & \OO_3   &
}$$

\begin{prop} \label{prop_closures}
\begin{enumerate}
    \item $Y_6\subset \Gr(2,S^2 U)$ is birational to a $\P^4$-bundle over $\P^2$.
    \item $Y_5\subset \Gr(2,S^2 U)$ is birational to a quadric bundle over $\P^2$ and is a non-normal projective variety with singular locus equal to $Y_4$.
    \item $Y'_4\simeq \P(U)\times \P(U^*)$ which is embedded in $\P^{14}$ via the linear series $\OO(2,1)$.
    \item The chain of embeddings 
\[ Y_3\subset Y_4 \subset \Gr(2,S^2 U)\subset \P(\wedge^2S^2 U) \cong \P^{14}
\] 
is isomorphic to the chain  
\[ B\subset \Hilb^2\P(U) \subset \Gr(2,S^2 U)\subset \P(\wedge^2S^2 U) \cong \P^{14}\]
discussed in the previous section.
\end{enumerate} 
\end{prop} 

\begin{proof} First, note that the variety $Y_6$ parametrizes pencils of conics 
containing a double line. We thus have a birational model 
$\widetilde{Y}_6\to Y_6$
parametrizing pairs $(P,\ell)$ of a pencil $P$ of conics on $U^*$ and a line $\ell\subset \P(U^*)$ such that $\ell^2$ belongs to $P$.  For the bundle $E=S^2U/\OO(-2)$ over $\P(U)$, we see that $\widetilde{Y}_6=\P(E)$, a $\P^4$-bundle over $\P^2$.

Next, $Y_5$ parametrizes pencils generated by a double line and a conic tangent to this line. Consider the birational model $\widetilde{Y}_5\to Y_5$ defined as the preimage of $Y_5$ in $\widetilde{Y}_6$. Let
$Q=U/\OO(-1)$ which is a bundle over $\P(U)$. 
Note that there is a natural morphism $E\rightarrow S^2Q$ and hence a natural rank three quadratic form 
$S^2E\rightarrow S^2(S^2Q)\rightarrow \det(Q)^2$. Then, $\widetilde{Y}_5\subset \widetilde{Y}_6=\P(E)$ is the corresponding quadric bundle over $\P(U)$. The morphism $\tilde{Y}_5\to Y_5$ is birational and is a double cover over $Y_4$. Thus, $Y_5$ is not normal along $Y_4$ and,
since the latter is the complement of the open orbit in $Y_5$, it has to be the singular locus.

The variety $Y'_4$ parametrizes pencils of reducible conics containing a fixed line $\ell$. Such a pencil is of the form $\ell L$ where $L$ is a hyperplane of $U$.  Hence, we have an isomorphism $Y'_4\cong\P(U)\times\P(U^*)$. The restriction of the tautological bundle of $\Gr(2,S^2 U)$ is the 
product of the rank one tautological bundle on $\P(U)$ 
with the rank two tautological bundle on the dual $\P(U^*)$.
Its determinant is $\OO(2,1)$ showing that the degree of $Y'_4$ is $(2h+h')^4=24$. 

In this description, we see that $Y_3\subset Y'_4=\P(U)\times\P(U^*)$ is the incidence quadric, the full flag variety of $U$. Its degree is $(h+h')(2h+h')^3=18$. 

To conclude, consider the variety $Y_4$ parametrizing pencils generated by two double lines and their degeneracies. Its dense open subset
$\OO_4$ is isomorphic to $\Sym^2\P(U)-\Delta$ where $\Delta$ denotes the diagonal. 
For dimension reasons, its closure $Y_4$ must be $\Hilb^2\P(U)\subset\Gr(2,S^2 U)$.
Finally, note that the diagonal in $\Hilb^2\P(U)=Y_4$ 
is $\P (T_{\P(U)})\cong F\cong Y_3$ as claimed which completes the proof.
\end{proof}

\subsection{Explicit equations}\label{sec:1.3}

\label{sec_proj}
Consider the group $\SL(U)\cong \SL_3(\C)$. 
Denote by $S_{a,b}$ the representation of $\SL(U)$ with highest weight $(a-b)\omega_1+b\omega_2$, where $\omega_1$ and $\omega_2$ are the fundamental weights and $a\geq b\geq 0$. 
Clearly, $S_{0,0}= \C$ is the trivial module, $S_{1,0} = U$ is the three-dimensional ``vector'' representation, and $S_{1,1}= U^*$ is its dual. 
Consider the $\SL(U)$-module $W = \wedge^2 S^2 U^*= S_{3,2}$
of dimension~$15$.
Its second symmetric square splits into irreducibles as
\begin{equation} S^2 W \cong S_{6,4}\oplus S_{4,3}\oplus S_{4,0}\oplus S_{3,1}\oplus S_{2,2}
\label{eq_dec} \end{equation}
of dimension $120=60+24+15+15+6$.
This decomposition can be easily checked computationally, e.g., via the {\tt SchurRings} \cite{Schur_M2}
package in {\tt Macaulay2} \cite{M2}.

Consider the $14$-dimensional projective space $\P W^* =\P (\wedge^2 S^2 U)\cong\P^{14}$ with coordinate ring $S^*(W)$. Regard elements of $S^2W$ as quadrics on $\P W^*$.
Using the decomposition~\eqref{eq_dec} above, 
define ideals of the ring $S^*(W)$ generated by spaces of quadrics as follows:
\begin{align}
\label{eqn:eqlabel}
\begin{split}
I_8 & =  \langle S_{3,1}\rangle \lhd S^*(W),\\
I_5 & =  \langle S_{3,1}, S_{2,2}\rangle \lhd S^*(W),\\
I_4 & =  \langle S_{3,1},  S_{2,2}, S_{4,3}\rangle \lhd S^*(W),\\
I_3 & =  \langle S_{3,1},  S_{2,2}, S_{4,3}, S_{4,0}\rangle \lhd S^*(W).
\end{split}
\end{align}
These ideals are generated by $15, 21, 45$ and $60$ quadrics, respectively.
The first ideal $I_8$ consists of quadrics parametrized by $S_{3,1}=\wedge^2 S^2 U$ itself
and it is well known that its vanishing locus is the Grassmannian: \[\V(I_8)=\Gr(2,S^2 U)\subset \P(\wedge^2S^2 U).\] 

\begin{prop} \label{prop_eqns} For $i=3,4,5$ we have, using the notations of the previous section, 
\[\V(I_i)\cong Y_i\subset\Gr(2,S^2 U)\subset \P(\wedge^2S^2 U) \cong \P^{14}.\]
In particular, the ideals of the orbit closures $Y_3,Y_4, Y_5 \subset \Gr(2,S^2 U)$ discussed before are generated by quadrics.
\end{prop}
\begin{proof} Our proof uses computational methods, 
with some details omitted.

Fix a basis $\{v_i\}$ of $U^*$.
The key to finding explicit equations is to derive an explicit form of the decomposition~\eqref{eq_dec} yielding 
basis elements for each of the modules on the right hand side in terms of the obvious basis of the left hand side $S^2 W \cong S^2\left(\wedge^2 S^2 U^*\right)$ 
consisting of symmetric pairs of elements of the form $[(v_i\otimes v_j)\wedge (v_k\otimes v_l)]$. 
This can be done using {\tt SLA}~\cite{deGraaf} in {\tt GAP}~\cite{GAP}.
We obtain explicit generators of these four ideals, which are listed in the Appendix. As expected, the $15$-dimensional space $S_{3,1}$ has a basis whose elements can readily be identified with the quadratic Pl\"ucker relations defining $\V(I_8)=\Gr(2,6)\subset\P^{14}$. 

Given the explicit polynomial generators, the dimension and degree
of the remaining ideals can be computed using {\tt Macaulay2}~\cite{M2}, namely:
$$\begin{array}{c|c|c|c}
\hbox{ideal} & \hbox{dimension} & \hbox{degree} 
& \hbox{Hilbert series} \\
\hline 
I_5 & 5 & 56 & (1+9t+24t^2+19t^3+3t^4)/(1-t)^6\\
I_4 & 4 & 21 & (1+10t+10t^2)/(1-t)^5\\
I_3 & 3 & 18 & (1+11t+6t^2)/(1-t)^4
\end{array}$$ 
We next claim that each of these ideals defines an irreducible and reduced subscheme in $\P^{14}$. 

We first verify that each top-dimensional component in each scheme
is irreducible of multiplicity~$1$.  To that end, 
certifiable witness point sets can be computed 
in numerical algebraic geometry using {\tt Bertini}~\cite{Bertini}
and {\tt alphaCertified}~\cite{alphaCertified}
via well-constrained subsystems~\cite{CertificationSquare}.
Each witness point set consists of 
degree-many points along a complimentary dimensional
linear space, which are nonsingular with respect
to the resulting system.
Certifiable monodromy loops
following~\cite{CertificationMonodromy}
yield that all of the witness points 
for each system are smoothly
connected, showing that 
each of the top-dimensional components is irreducible of multiplicity $1$
with respect to the corresponding ideal.

In each case, showing that 
the scheme is irreducible and reduced is now equivalent to showing that the top-dimensional irreducible component has the same Hilbert series as the entire scheme; this precludes the existence of embedded components.
This was verified using \cite{aCM} 
with certified numerical Hilbert function computations~\cite{AltProblem}.
In fact, this also showed that each scheme
is arithmetically Cohen-Macaulay via \cite{aCM}.

The ideals $I_5\subset I_4 \subset I_3$ thus define $\SL(3)$-invariant 
irreducible reduced subschemes of $\V(I_8)=\Gr(2,6)\subset\P^{14}$ of dimensions
$5, 4, 3$ respectively. As there is only one five-dimensional orbit closure in $\Gr(2,6)$ by our results in the previous section, we must have
$\V(I_5)=Y_5\subset\P^{14}$, and then necessarily 
$\V(I_4)=Y_4\supset \V(I_3)=Y_3$ as claimed. 
\end{proof}

\begin{remark} The $45$ quadrics defining $Y_4\cong\Hilb^2\P^2 \subset \Gr(2,6)$ were also determined, using a different method, in~\cite[Sect. 2]{St}. That paper also studied the corresponding tropicalization. 
\end{remark}

    \begin{remark} In the language of the previous section, the quadrics in $S_{3,1}$ and $S_{4,0}$ 
    define a four-dimensional variety of degree $24$ inside
    of $\Gr(2,S^2 U)$ which is $Y_4'\cong \P^2\times\P^2$ embedded by ${\mathcal O}(2,1)$~as~above. \end{remark}

\subsection{A different way to derive equations for orbit closures}\label{sec:1.4}

We explain here an alternative, synthetic way to re-derive the quadratic equations of the orbit closures obtained above by a computer-based calculation. In this section, we will use the language of $\GL(U)$-modules to respect the full symmetry of the problem. By a slight abuse of notation, we will use $\psi(-)$ as shorthand for the quadratic form associated to a symmetric bilinear form $\psi(-,-)$.

\smallskip
\noindent {\bf The equations of $Y_5$.} The equations of $Y_5\subset \P^{14}$ are the Pl\"ucker equations together with another irreducible module of quadratic equations
that is $\GL(U)$-isomorphic to 
$S^2 U^*\otimes \det (U^*)^2$. Those equations have a simple description in terms of the discriminant of ternary quadratic forms (or symmetric tensors), which is a $\GL(U)$-equivariant map 
$$\delta : S^3(S^2 U)\rightarrow \det (U)^2.$$
Polarizing yields a morphism \[\Delta : S^2(S^2 U)\rightarrow S^2U^*\otimes\det (U)^2.\]
Using the same trick for $U^*$ and twisting by $\det (U)$ appropriately yields another $\GL(U)$-equivariant map
$$\Delta^* : S^2(S^2 U^*\otimes\det (U)^2)\rightarrow S^2U\otimes\det (U)^2.$$

In order to describe the equations of $Y_5\subset\Gr(2,S^2 U)$, since it is the unique irreducible $\GL(U)$-component of $S^2(\wedge^2 S^2 U^*)$ isomorphic to  $S^2 U^*\otimes \det (U^*)^2$, it suffices to exhibit 
a nonzero morphism 
\[\Psi\colon S^2(\wedge^2 S^2 U)\to S^2 U\otimes \det (U)^2\] 
from the dual of the former to the dual of the latter. For $q_1, q_2, q_3, q_4$ belonging to $S^2U$, 
we claim that the following formula defines 
such a morphism:
$$\Psi(q_1\wedge q_2, q_3\wedge q_4) = \Delta^*(\Delta(q_1,q_3), \Delta(q_2,q_4))-
 \Delta^*(\Delta(q_1,q_4), \Delta(q_2,q_3)).$$
Since this is nonzero (see below) and has all the required properties,
it expresses in a compact form 
all the equations of $Y_5$ apart from the Pl\"ucker relations. In the following statement, we use the same notations~$\Psi$ and $\Delta$ for the quadratic forms associated to the symmetric bilinear forms defined by these symbols. 

\begin{prop}
The subvariety $Y_5\subset\Gr(2,S^2 U)$ of the Grassmannian $\Gr(2,S^2 U)$ is defined by
$$\Psi(q_1\wedge q_2) = \Delta^*(\Delta(q_1), \Delta(q_2))-\Delta^*(\Delta(q_1,q_2))=0$$
for $q_1, q_2\in S^2U$.
\end{prop}

\medskip
As a sanity check, let us evaluate $\Psi$ on the representatives of the $\PGL_3$-orbits of
$G(2,S^2U)$. For this we normalize $\Delta$ by letting $\Delta(u^2,v^2)=(u\wedge v)^2$. 
\begin{itemize} 
\item For $\mathcal{O}_5$,
we let $q_1=x^2$ and $q_2=y^2+xz$.
Hence, $\Delta(q_1)=0$ and $\Delta(q_1,q_2)=(x\wedge y)^2$
has rank one, so that $\Psi(q_1\wedge q_2) =-\Delta^*(\Delta(q_1,q_2))=0$. 
\item For $\mathcal{O}_6$, we let $q_1=x^2$ and $q_2=yz$. 
Hence, $\Delta(q_1)=0$ and writing $4q_2=(y+z)^2-(y-z)^2$,
we get $\Delta(q_1,q_2)=(x\wedge y)(x\wedge z)$. Since this has rank two, $\Psi(q_1\wedge q_2) =-\Delta^*(\Delta(q_1,q_2))\ne 0$. 
\item For $\mathcal{O}'_4$, we let $q_1=xy$, $q_2=xz$,
$u=x\wedge y$, and $v=x\wedge z$. Then, \mbox{$2\Delta(q_1)=-u^2$}, 
\mbox{$2\Delta(q_2)=-v^2$}, and $2\Delta(q_1,q_2)=-uv$. 
Therefore, $\Delta^*(\Delta(q_1), \Delta(q_2))=\Delta^*(u^2, v^2)/4=(u\wedge v)^2/4$ together with $\Delta^*(\Delta(q_1,q_2))=\Delta^*(uv)/4=-(u\wedge v)^2/8$
yields $\Psi(q_1\wedge q_2)=3(v\wedge w)^2/8\ne 0$.
\end{itemize}

\bigskip\noindent {\bf The other quadratic equations.}
Consider the remaining modules of equations defining $Y_4, Y_3\subset\P^{14}$. These irreducible modules are generated by highest weight vectors. We can write down explicit quadratic polynomials on the Grassmannian in terms of those highest weight vectors as follows. 

The $15$-dimensional module is the irreducible $\GL(U)$-module $S^4(\wedge^2U^*)$, whose highest weights vectors are the tensors of the form $(e\wedge f)^4$ for $e,f\in U^*$. 
We want to associate to such a vector
a quadratic polynomial $P$ on the cone of tensors of the form $q_1\wedge q_2$ for $q_1,q_2\in S^2U$. 
In other words, we need to find a polynomial in $e,f,q_1,q_2$, which has degree four in $e$ and $f$, degree two in $q_1$ and $q_2$, 
invariant under the action of~$\GL(U)$, 
and compatible with the skew-symmetry conditions. 
By the Fundamental Theorems of Invariant Theory, 
such an invariant 
polynomial has to be expressed 
in terms of contractions of $q_1$ and $q_2$ by $e$ and~$f$. 
A straightforward computation shows that up to scalar, there is only one possibility, namely
$$P(q_1\wedge q_2)= \Big( q_1(e)q_2(f)-q_2(e)q_1(f)\Big) ^2 +4\Big( q_1(e,f)q_2(f)-q_2(e,f)q_1(f)\Big)
\Big( q_1(e,f)q_2(e)-q_2(e,f)q_1(e)\Big).$$
One can easily test 
this polynomial on representatives of the orbits in $\Gr(2,S^2U)$ and check that it vanishes identically only on the orbit $\mathcal{O}_3=Y_3$.

The second, $24$-dimensional module is the $\GL(U)$-submodule of $S^3U^*\otimes \wedge^2 U^*\otimes \wedge^3 U^*$
whose highest weights vectors are the tensors of the form $e^3(e\wedge f)(e\wedge f\wedge g)$
for $e,f,g\in U^*$. Again, we need to associate to 
such a vector a quadratic polynomial $Q$ on the cone of tensors of the form $q_1\wedge q_2$ for $q_1,q_2\in S^2U$. 
Thus, we need to find a polynomial in $e$, $f$, $g$, $q_1$, and $q_2$ which has degree five in $e$,
degree two in $f$, degree one in $g$, 
degree two in $q_1$ and $q_2$, 
invariant under the action of $\GL(U)$, 
and compatible with the skew-symmetry conditions. Another 
straightforward computation shows that up to scalar, there is only one possibility, namely
$$\begin{array}{rl}
\hbox{\small $Q(q_1\wedge q_2) =$} & 
\hbox{\small $\Big(q_1(e,g)q_2(e)-q_2(e,g)q_1(e)\Big)\Big(q_1(e)q_2(f)-q_2(e)q_1(f)\Big)$}~+ \\
& \hbox{\small $\Big( q_1(e,f)q_2(e)-q_2(e,f)q_1(e)\Big)\Big( q_1(f,g)q_2(e)-q_2(f,g)q_1(e)+q_1(e,g)q_2(e,f)-q_2(e,g)q_1(e,f)\Big).$}
\end{array}$$
Again, one can easily test this polynomial on 
representatives of the orbits in $\Gr(2,S^2U)$ and check 
that it vanishes identically only on $\mathcal{O}_3$ and $\mathcal{O}_4$. We deduce

\begin{prop}
The subvarieties $Y_3\subset Y_4\subset\Gr(2,S^2 U)$ of the Grassmannian $\Gr(2,S^2 U)$ are defined by the sets of equations
$$\Psi(q_1\wedge q_2) = Q(q_1\wedge q_2)=0$$
and 
$$\Psi(q_1\wedge q_2) = P(q_1\wedge q_2)=Q(q_1\wedge q_2)=0$$
for $q_1, q_2\in S^2U$, respectively.
\end{prop}

\section{The punctual Hilbert scheme of four points on \texorpdfstring{$\A^3$}{A3}}

\subsection{Non-planar clusters of length four on affine three-space}\label{sec:2.1}

Recall the fixed three-dimensional vector space $U$ from the previous section. In this section, we will think of its dual $T=U^*$ as a copy of affine $3$-space $\A^3$ with ring of functions $\C[T]=\Sym^\bullet U$. 

Let $\Hilb^m(T)\cong \Hilb^m(\A^3)$ denote the Hilbert scheme of~$m$ points on affine three-space. The additive structure on $T$ yields a 
center-of-mass morphism $c\colon \Hilb^m(T)\to T$. 
As is true in all dimensions, $\Hilb^m(T)$ is a nonsingular variety for $m\leq 3$. 
It is also known~\cite{katz} that $\Hilb^4(T)$ is an irreducible and reduced variety of dimension 12, singular along the locus of length-four subschemes of~$T$ given by the squares $[{\mathbf m}_p^2]\in \Hilb^4(T)$ of maximal ideals of points $p\in T$. It has a dense affine open subset $\Hilb^4(T)_\np$ containing all its singularities
defined by the condition that the clusters
parametrized by its points are non-planar,
i.e., not scheme-theoretically contained in a plane. 

Although the following was already proved in~\cite{DSz},
we present a variant of the proof suited to the present narrative.

\begin{theorem}\label{theorem_hilbC4}
There is an $\SL(U)$-equivariant isomorphism \[\Hilb^4(T)_\np \cong T\times {\mathcal C}\Gr(2,S^2U),\]
where ${\mathcal C}\Gr(2,S^2 U)\subset \wedge^2S^2 U$ is the affine cone over the Grassmannian 
\mbox{$\Gr(2,S^2U)\subset  \P(\wedge^2S^2 U) \cong \P^{14}$}.
\end{theorem}

\proof If $I_\xi\lhd \C[T]=\Spec\Sym^\bullet U$ 
is the ideal corresponding to a point $\xi\in \Hilb^4(T)_{\np}$, 
then the map \mbox{$\C\oplus U\rightarrow \C[T]/I_\xi$} 
is an isomorphism of vector spaces. 
The resulting algebra structure on the vector space $\C\oplus U$ is encoded by two symmetric bilinear maps:
$$a : U\otimes U\rightarrow \C \qquad \mathrm{and} \qquad m : U\otimes U\rightarrow U.$$
Requiring the product on $\C\oplus U$ to be associative leads to the following equations for any $x,y,z\in U$:
\begin{eqnarray}
\label{c1} & & a(x,m(y,z)) = a(y, m(x,z)), \\
\label{c2} & & m(x,m(y,z)) +a(y,z)x= m(y, m(x,z))+a(x,z)y. \end{eqnarray}

Using the highest weight notation for $\SL_3$-modules introduced in the previous section, first observe that $m$ is a tensor in 
$S^2U^*\otimes U\cong U^*\oplus S_{3,2}$
and the projection to $T=U^*$ is the center of mass map~$c$. Using a translation by the evident action of $U^*$ on the whole setup, it suffices to restrict to the case where the center of mass of the ideal $I_\xi$ is at the origin in $T$ so that $m\in S_{3,2}$. 

Equation (\ref{c2}) shows that $a(y,z)x-a(x,z)y$ is determined by $m$, $x$, $y$, and $z$ 
which implies that $a$ is uniquely determined by $m$ 
and depends quadratically on it. As confirmed by~\cite{Schur_M2}, there exists a unique $\SL(U)$-equivariant map up to scale, namely
$$\Theta :  S^2(S_{3,2}) \rightarrow S^2 U^*.$$
Fixing the correct normalization, equation (\ref{c2}) implies that $a=\Theta(m)$. 
Now, for any  $x,y,z\in U$, we can rewrite 
equations~\eqref{c1}-\eqref{c2} in the following form: 
\begin{eqnarray}
\label{c3} & & \Theta(m)(x,m(y,z)) = \Theta(m)(y, m(x,z)),\\ \label{c4} & & m(x,m(y,z)) +\Theta(m)(y,z)x= m(y, m(x,z))+\Theta(m)(x,z)y.\end{eqnarray}
These equations are families of cubic and 
quadratic equations on $m\in S_{3,2}$. 
We claim that
\begin{enumerate}
\item[(a)] the quadratic equations (\ref{c4}) on $m$ are equivalent to the Pl\"ucker equations on $\wedge^2S^2U$;
\item[(b)] the cubic equations (\ref{c3}) are implied by the quadratic ones.
\end{enumerate}

In order to prove (a), 
recall the $\SL_3$-decomposition
\begin{equation} S^2 S_{3,2} \cong S_{6,4}\oplus S_{4,3}\oplus S_{4,0}\oplus S_{3,1}\oplus S_{2,2}
\label{eq_dec_2}
\end{equation}
already used above in~\eqref{eq_dec}.
Equation (\ref{c4}) asks for the vanishing of a cubic tensor in $x$, $y$, and $z$, 
skew-symmetric in $x$ and $y$, and takes values in $U$,
i.e., an element of the $\SL_3$-decomposition
\begin{equation}
\label{last_dec}
{\mathrm{Hom}}(\wedge^2 U\otimes U,U)\cong 2S_{1,1}\oplus S_{3,1}\oplus S_{2,2}.\end{equation}
Comparing~\eqref{eq_dec_2} with~\eqref{last_dec}, the common terms are the last two irreducible components; they are 
the only ones imposing non-trivial conditions. The component $S_{2,2}$ has already been taken into account by letting $a= \Theta(m)$. The remaining conditions are the quadratic equations
parametrized by $S_{3,1}$.
This means that we get the quadratic equations spanning the
ideal $I_8\lhd S^*(\wedge^2S^2U^*)$ from~\eqref{eqn:eqlabel}, parametrized 
by $\wedge^2 S^2 U$, which are precisely the Pl\"ucker equations.

We conclude that $\Hilb^4(T)_\np$ is contained in $T\times {\mathcal C}\Gr(2,S^2U)$ which implies, by a dimension count, that these two reduced schemes must be equal. 
Since the ideal of ${\mathcal C}\Gr(2,S^2U)$ is radical, we also deduce claim (b): the cubic relations~\eqref{c3} do not impose any further conditions on $m$. 
\qed

\subsection{The space of non-planar punctual clusters}\label{sec:2.2}

Let $\Hilb^m_0(T)\subset \Hilb^m(T)$ denote the punctual Hilbert scheme of $m$ points, the subscheme of the Hilbert scheme  $\Hilb^m(T)$ given by the condition that the support of the zero-dimensional subscheme being parametrized is at the origin $0\in T$. It carries a natural projective scheme structure as the scheme-theoretic fiber of the Hilbert--Chow morphism $\Hilb^m(T)\to S^m(T)$ over $m\cdot 0\in S^m(T)$, but here we consider it in its reduced scheme structure. 

Clearly, $\Hilb^2_0(T)\cong\P(T)$,
and $\Hilb^3_0(T)$ and $\Hilb^4_0(T)$ are known to be irreducible but singular projective varieties of dimensions $4$ and $6$ respectively. Descriptions of these spaces as well as natural desingularizations are given in~\cite{Tik} from a sheaf-theoretic perspective. 
We will describe the affine open set  $\Hilb^m_0(T)_{\np}\subset\Hilb^4_0(T)$ obtained by intersecting
$\Hilb^4_0(T)$ with the set of non-planar clusters.
Note that $\Hilb^m_0(\A^3)_{\np}$ is dense in $\Hilb^4_0(\A^3)$ and forms an affine neighbourhood of its most interesting point 
$[{\mathbf m}_0^2]\in \Hilb^4_0(\A^3)$, the length-four subscheme of $\A^3$ given by the square of the maximal ideal of the origin $0\in \A^3$. The following is our main result.

\begin{theorem} \label{main_thm}
The reduced space $\Hilb^4_0(T)_{\np}$ of non-planar, punctual clusters of length $4$ on $T$ is $\SL(U)$-equivariantly isomorphic to the cone $\SC Y_5\subset\wedge^2S^2 U\cong\A^{15}$ over the projective variety $Y_5\subset\P(\wedge^2S^2 U)$ described in Propositions~\ref{prop_closures}-\ref{prop_eqns}. In particular, $\Hilb^4_0(T)_{\np}$ is a 
non-normal subvariety of $\A^{15}$ cut out by $21$ explicitly computable quadrics
and has a codimension one singular locus isomorphic to the affine cone $\SC \Hilb^2(\P^2)\subset \wedge^2S^2 U$
with the apex of the cone
corresponding to the distinguished ideal 
\hbox{$[{\mathbf m}_0^2]\in \Hilb^4_0(T)$}.
\end{theorem}

\proof The classification of $\PGL(U)$-orbits in $\Gr(2,U)$ explained in Section~\ref{subsec_pencils} above shows that $\SC Y_5$ is the only six-dimensional $\SL(U)$-stable subvariety of $\SC\Gr(2,U)$. \qed

\begin{remark} The description of the singular locus of (a neighbourhood of $[{\mathbf m}_0^2]$ in) $\Hilb^4(T)_0$ is not a new result as it was also obtained in~\cite{Tik} using sheaf-theoretic methods. The main advantage of our approach is that we can describe all these spaces by explicit affine equations which may be useful in applications~\cite{CRHS}.\end{remark}

\begin{remark} It is also possible to prove this result using the language used in~\cite{DSz} followed by some computer calculations. In coordinates, we can represent a non-planar cluster of length four using three four-by-four matrices $\phi_1$, $\phi_2$, and $\phi_3$ 
which describe the action of the coordinate functions of $T$
on the 4-dimensional vector space $\C\oplus U$. The conditions for a triple $(\phi_i)$ to describe a cluster 
become explicit equations which reduce to the Pl\"ucker relations. To describe a punctual cluster based at~$0$, the matrices $\phi_i$ should additionally be nilpotent. It can be checked computationally that the reduced ideal of conditions arising from $\Tr \wedge^k\phi_i=0$ for $1\leq k \leq 4$ are generated by the linear relations $\Tr\phi_i=0$ and the quadratic conditions $\Tr \wedge^2\phi_i= \Tr\wedge^2(\phi_i+\phi_j)=0$, the latter being the $6$ extra quadrics of $I_5$ defining~$Y_5$ inside $\Gr (2, S^2 U)$ in Proposition~\ref{prop_eqns} above. 
\end{remark}

\begin{remark} The orbit decomposition of $\Gr(2,S^2U)$
from Section~\ref{subsec_pencils} yields a decomposition of~${\mathcal C}\Gr(2,S^2U)$. It is easy to check that the various orbits correspond to different length four subschemes in $T$ as follows. 
\begin{enumerate}
\item ${\mathcal C}O_8$ parametrizes four general points;
\item ${\mathcal C}O_7$  parametrizes two reduced general points $p,q\in T$ and a degree two scheme at $-(p+q)/2$;
\item ${\mathcal C}O_6$  parametrizes two degree two schemes supported on opposite general points;
\item ${\mathcal C}O'_6$ parametrizes one reduced point at $p\in T$ and a degree three scheme supported on $-p/3$;
\item ${\mathcal C}O_5$  parametrizes an open subset of $\Hilb^4_0(T)$;
\item ${\mathcal C}O_4$  parametrizes clusters with normal form  $xy=z$, all other products being equal to zero;
\item ${\mathcal C}O'_4$ parametrizes one reduced point $p$ and a fat point in a plane at $-p/3$; 
\item ${\mathcal C}O_3$  parametrizes clusters with normal form $x^2=z$, all other products being equal to zero.
\end{enumerate} 
\end{remark}

\subsection{A concrete computation}\label{sec:2.3}

Starting from a non-planar scheme with center of mass $0$ corresponding to an ideal $I\lhd \C[T]$, recall that 
we get an algebra structure on the vector space $\C\oplus U$ which is 
partly encoded by a symmetric bilinear multiplication map \[m\in \mathrm{Hom}(S^2 U,U)\cong S^2U^*\otimes \wedge^2U^*\otimes \det(U).\] Apply to $m$ the morphism 
\[\begin{array}{rcccl}
\Gamma & :  & S^2U^*\otimes \wedge^2U^*& \rightarrow & \wedge^2(S^2U^*)\\
&& a^2\otimes b\wedge c & \mapsto & 
 ab\wedge ac.\end{array}\]
Then, the discussion above shows that $\Gamma (m)\in \wedge^2(S^2U^*)\otimes \det(U)$
must be a decomposable tensor.
 
To see this in a concrete example, fix a basis $\{x,y,z\}$ of $U$ and consider the reduced scheme \[\zeta= \{(-1,0,0), (0,-1,0), (0,0,-1), (1,1,1)\}\subset T.\] 
In the ring $\C[x,y,z]/I_\zeta$, we have 
$$ x^2 = \frac{1}{2}(1-x+y+z), \quad y^2 = \frac{1}{2}(1+x-y+z), \quad z^2 = \frac{1}{2}(1-x+y-z), $$
$$ xy = yz =zx = \frac{1}{4}(1+x+y+z).$$
\smallskip
The tensor $m$ is obtained
by keeping the degree one part of the right hand side in these equations.
In terms of the dual basis $\{e,f,g\}$ of $U^*$, 
note that the basis of $S^2U^*$ dual to the basis $\{x^2,y^2,z^2,yz,xz,xy\}$ of $S^2U$ is $\{e^2,f^2,g^2,2fg,2eg,2ef\}$. So, as a tensor, 
$$ \begin{array}{rcl}
4m & = & 2e^2\otimes (-x+y+z)+2f^2\otimes (x-y+z)+2g^2\otimes (x+y-z)+2(fg+ge+ef)\otimes (x+y+z) \\
 & = & (e^2+f^2+g^2+(e+f+g)^2)\otimes (x+y+z)-4(e^2\otimes x+f^2\otimes y+g^2\otimes z).
 \end{array}$$
Note that up, to a common factor, $x$ identifies with $f\wedge g$, $y$ with $g\wedge e$ and $z$ with
$e\wedge f$. Substituting these expressions and applying $\Gamma$, we get, after letting $h=e+f+g$, 
$$ \begin{array}{rcl}
4\Gamma (m) &= & ef\wedge eg+eg\wedge e^2+e^2\wedge ef+
f^2\wedge fg+fg\wedge fe+fe\wedge f^2+gf\wedge g^2+g^2\wedge ge+ge\wedge gf \\ 
  & & \hspace*{2cm}+hf\wedge hg+hg\wedge he+he\wedge hf-4(ef\wedge eg+fg\wedge fe+ge\wedge gf) \\
    & =& hf\wedge hg+hg\wedge he+he\wedge hf-(ef\wedge eg+fg\wedge fe+ge\wedge gf) \\
    & & \hspace*{3cm}+he\wedge (ef-eg)+hf\wedge (fg-ef)+hg\wedge (eg-fg), 
      \end{array}$$
      and the final result of our computation is 
      $$4\Gamma (m)=(he-hf +eg-fg   )\wedge (hf-hg+ef-eg).$$

  This is, as expected, a decomposable tensor. Note that since the $\GL_3$-orbit of our scheme is open
  in the subvariety of $\Hilb^4(T)$ parametrizing schemes with center of mass at the origin, this yields another proof of Theorem~\ref{theorem_hilbC4}. 
  
  Moreover, observe that we can rewrite $he-hf +eg-fg =e^2-f^2+2eg-2fg=(e+g)^2-(f+g)^2$, and 
  similarly $hf-hg+ef-eg= (e+f)^2-(f+g)^2$, so that 
    $$4\Gamma (m)=(e+f)^2\wedge (f+g)^2+(f+g)^2\wedge
    (g+e)^2+(g+e)^2\wedge (e+f)^2.$$
  This leads to a down-to-earth interpretation of the 
  map 
  \[\pi\colon \Hilb^4(T)_\np \mapsto {\mathcal C}\Gr(2,S^2U)\]
  from Theorem~\ref{theorem_hilbC4} on the open set of reduced subschemes.

  \begin{prop} \label{prop_downtoearth}
  Consider a finite subscheme $\zeta\subset T$ consisting of four 
  non-coplanar reduced points $p_i\in T$ for $i=1,\dots,4$
  with center of mass $p_0$. The squares of the six vectors $p_{ij}=p_i+p_j-2p_0\in T$ are three tensors 
  $p_{12}^2=p_{34}^2$, $p_{13}^2=p_{24}^2$ and $p_{14}^2=p_{23}^2$ in $S^2T=S^2U^*$ and
  $$\pi(\zeta)= (p_{12}^2\wedge p_{13}^2+p_{13}^2\wedge p_{14}^2+p_{14}^2\wedge p_{12}^2)\otimes \omega^{-1}$$
  where the twist  $\omega\in \det(U^*)$ is given by 
  $$\omega = p_1\wedge p_2\wedge p_3- p_2\wedge p_3\wedge p_4+p_3\wedge p_4\wedge p_1-p_4\wedge p_1\wedge p_2.$$
  \end{prop}
  
Note that  $\omega$ is invariant under a common translation of the four points and is non-zero exactly when $\zeta\subset T$ is non-planar. Moreover, permuting the four points multiplies $\omega$ by the sign of the permutation but $\pi(\zeta)$ itself remains invariant.
Thus, it only depends on $\zeta$ and not on the order of the four points.  

\begin{proof}[Proof of Proposition~\ref{prop_downtoearth}]
The expression for $\pi(\zeta)$ depends equivariantly 
on $\zeta$ and yields the correct expression when 
$$\zeta = (-e,-f,-g,e+f+g).$$ 
Since the orbit of $\zeta$ in  $\Hilb^4(T)_\np$ is dense, this expression must be correct everywhere. 
\end{proof} 

\section*{Appendix. Explicit polynomials}

The following set of {\tt Macaulay2} commands generates the ideals $I_8, I_5, I_4$ and $I_3$ in~\eqref{eqn:eqlabel}. As proved in the main body of the paper, the ideal $I_4$ defines the projective image of the embedding~\eqref{emb1}, whereas $I_5$ defines the affine image of the embedding~\eqref{emb2}.

\vspace{0.1in}
{
\tt

\small

\noindent S=QQ[a,b,c,d,e,f,g,h,i,j,k,l,m,n,o];

\vspace{0.1in}

\noindent 
I8=ideal(a*j-b*g+c*f, a*k-b*h+d*f, a*l-b*i+e*f, a*m-c*h+d*g, a*n-c*i+e*g, a*o-d*i+e*h, 
\newline b*m-c*k+d*j, b*n-c*l+e*j, b*o-d*l+e*k, c*o-d*n+e*m, f*m-g*k+h*j, f*n-g*l+i*j, 
\newline f*o-h*l+i*k, g*o-h*n+i*m, j*o-k*n+l*m);

\vspace{0.1in}

\noindent I5=I8+ideal(2*d*o-e*n-2*f*o-2*h*l+i*i-2*j*l+3*k*k, 2*a*i-2*a*k-2*b*h+2*b*j-c*e+d*d+3*f*f,\newline c*n-2*d*m-2*f*m-2*g*i-2*g*k+3*h*h+j*j, a*n-2*b*m-c*k+d*h+d*j-e*g+f*h-f*j, 
\newline 2*a*o-b*n-c*l+d*i+d*k-e*h+f*i-f*k,  c*o-e*m-f*n-2*g*l+h*i+h*k-i*j+j*k);

\vspace{0.1in}

\noindent I4=I5+ideal(a*d+a*f-b*c, a*e-b*d+b*f, g*n-h*m-j*m,   c*m-g*h+g*j, e*o-i*l+k*l, i*o+k*o-l*n, 
\newline
3*a*i+a*k-b*h-3*b*j-2*d*f, 2*a*k+b*h-3*b*j-3*c*e+3*d*d-d*f-6*f*f,  2*a*l+b*i-3*b*k-e*f,  \newline 2*a*m+c*h-d*g-3*f*g,  a*n+2*b*m-c*i+c*k+3*d*h-d*j-2*e*g-6*f*h, 2*b*o+d*l-e*k-3*f*l,  \newline 2*a*o+3*b*n+d*i+3*d*k-e*h-3*e*j-6*f*i-6*f*k,  2*b*m-3*c*i+c*k+6*d*h-d*j-3*e*g-6*f*h+3*f*j,
\newline b*n-c*l+3*d*k-2*e*j-3*f*i,  3*c*n-9*d*m+5*f*m+12*g*i-2*g*k-6*h*h-7*h*j+3*j*j, \newline c*o+2*d*n+e*m-3*h*i+3*j*k, 3*d*m-f*m-3*g*i+g*k+2*h*j, 2*d*o+e*n-h*l-i*i-i*k+j*l+2*k*k, \newline  3*d*n+3*e*m-f*n-2*g*l-6*h*i+6*h*k-i*j+3*j*k,  3*e*n-2*f*o-h*l-3*i*i+i*k-3*j*l+6*k*k, \newline 2*g*o+h*n-i*m-3*k*m,  3*h*o+j*o-k*n-2*l*m, 3*a*h-a*j-2*b*g-c*f);

\vspace{0.1in}

\noindent I3=I4+ideal(4*a*g-c*c, 4*b*l-e*e, 4*m*o-n*n, 2*a*h+a*j+b*g-c*d, a*l+b*i+2*b*k-d*e, \newline a*m+c*j-d*g+2*f*g, b*o+d*l-e*i+2*f*l, g*o-i*m-j*n+2*k*m, 2*h*o-i*n-j*o+l*m, 
\newline 2*a*i+4*a*k+4*b*h+2*b*j-c*e-2*d*d, 
c*n-2*d*m+4*f*m-2*g*i+4*g*k-2*j*j, \newline 2*d*o-e*n+4*f*o+4*h*l-2*i*i-2*j*l, a*n+b*m+2*c*k-2*d*h+d*j-e*g+4*f*h+2*f*j,
\newline a*o+b*n+c*l-d*i+2*d*k-2*e*h+2*f*i+4*f*k, c*o-e*m+2*f*n+g*l-2*h*i+4*h*k-i*j-2*j*k);

}

\subsection*{Acknowledgements } 
The authors would like to thank Bernd Sturmfels for asking the question that lead to this work, Heather Harrington for introducing some of the collaborators to each other, and Willem de Graaf, Jack Huizinga, Miles Reid, Tim Ryan, and Anna Seigal for helpful advice and correspondence. 
J.D.H. and B.Sz. acknowledge support from NSF grant CCF-1812746 and
EPSRC grant EP/R045038/1 respectively.

\end{document}